\documentclass[reqno]{amsart}
\usepackage{amssymb}
\usepackage{amsmath}
\usepackage{url}

\usepackage{bbm}

\makeatletter
\@addtoreset{equation}{section}
\makeatother

\renewcommand\thefigure{\thesection.\@arabic\c@figure}
\renewcommand\thetable{\thesection.\@arabic\c@table}

\newtheorem{theorem}{Theorem}

\newtheorem{proposition}[theorem]{Proposition}

\newtheorem{cor}[theorem]{Corollary}

\newcommand{\bb}[1]{{\mathbb #1}}

\newcommand{\eps}{\varepsilon}

\def\R{\mathbb R}
\def\N{\mathbb N}

\def\N{\mathbb N}

\def\T{\bb T}

\def\A{A}
\def\S{\mathcal T}

\begin{document}

\author{Sunder Sethuraman and Shankar C. Venkataramani}

\address{\noindent Department of Mathematics, University of Arizona,
  Tucson, AZ  85721
\newline
e-mail:  \rm \texttt{sethuram@math.arizona.edu}
}

\address{\noindent Department of Mathematics, University of Arizona,
  Tucson, AZ  85721
\newline
e-mail:  \rm \texttt{shankar@math.arizona.edu}
}

\title[On superlinear preferential attachment]{On the growth of a superlinear preferential attachment scheme}

\begin{abstract}
We consider an evolving preferential attachment random graph model where at discrete times a new node is attached to an old node, selected with probability proportional to a superlinear function of its degree.  For such schemes, it is known that the graph evolution condenses, that is a.s. in the limit graph there will be a single random node with infinite degree, while all others have finite degree.

In this note, we establish a.s. law of large numbers type limits and fluctuation results, as $n\uparrow\infty$, for the counts of the number of nodes with degree $k\geq 1$ at time $n\geq 1$.  These limits rigorously verify and extend a physical picture of Krapivisky, Redner and Leyvraz (2000) on how the condensation arises with respect to the degree distribution.
 \end{abstract}



\maketitle

\section{Model and Results}
There has been much recent interest in preferential attachment schemes which grow a random network by adding, progressively over time, new nodes to old ones based on their connectivity (cf. books and surveys \cite{Albert-Barabasi-02}, \cite{Cald}, \cite{Chung-Lu}, \cite{Durrettbook}, \cite{Mitzenmacher}, \cite{Newman10}).  Part of the interest is that such schemes, which might be seen as forms of reinforcement dynamics, have interesting and complex structure, relating to some `real world' networks.

Consider the following model.  At time $n=1$, an initial network $G_1$ is composed of two vertices connected by a single edge.  At times $n\geq 2$, a new vertex is attached to a vertex $x\in G_{n-1}$ with probability proportional to
$w(d_x(n))$, that is with chance $w(d_x(n-1))/\sum_{y\in G_{n-1}}w(d_y(n-1))$, where $d_z(j)$ is
the degree at time $j$ of vertex $z$ and $w=w(d): \N \rightarrow \R_+$ is a
`weight' function in the form $w(d) = d^\gamma$ for a fixed $\gamma>-\infty$ .  In this way, a random tree is grown.

Let now ${Z}_k(n)$ be the number of vertices in $G_n$ with degree $k$, that is
${Z}_k(n) = \sum_{y\in G_{n-1}} 1(d_y(n) = k)$.  Also, let $S(n) = \sum_{k=1}^nw(k)Z_k(n)$ be the total weight of $G_n$ and $V(n) = \max\{k\ | Z_k(n) \geq 1\}$ denote the largest degree among the vertices in $G_n$. By construction, $V(n) \leq n$.  Let also $\Phi_k(n) = \sum_{\ell\geq k} Z_\ell(n)$ be the count of vertices with degree at least $k$ in $G_n$.

We observe that the total number of vertices and edges in $G_n$ are $n+1$ and $n$ respectively, so that we have the `conservation laws'
 \begin{align}
 \label{conservation}
 \sum_{k= 1}^n Z_k(n)  & = n+1, \ \ {\rm and } \ \ 
 \sum_{k= 1}^n k Z_k(n)   = 2 n 
\end{align}
 
In \cite{KRL} (see also \cite{KR}), a trichotomy of growth behaviors was observed depending on the strength of the exponent $\gamma$.
\medskip

{\it Linear case.}  First, when $\gamma =1$, the scheme is often referred to as the `Barabasi-Albert' model.  Here, the degree structure satisfies, for $k\geq 1$,
 $$\lim_{n\rightarrow\infty} \frac{{Z}_k(n)}{n} \ = \ \frac{4}{k(k+1)(k+2)} \ \ {\rm a.s.}$$
This limit was first proved in mean-value in \cite{DM}, \cite{KRL}, in probability in \cite{BRST-01}, and almost-surely via different methods in \cite{AGS}, \cite{CS}, \cite{Mori-01}, \cite{RTV}.

\medskip

{\it Sublinear case.} Next, when $\gamma<1$,
it was shown that
\begin{equation*}
\lim_{n\uparrow\infty}\frac{{Z}_k(n)}{n} \ =\  q(k) \ \ \ {\rm a.s.}\end{equation*}
 Although
$q$ is not a
power law, it is in form where it decays faster than any polynomial
\cite{CSV}, \cite{RTV}: For $k\geq 1$,
\begin{equation*}
\label{stretched}
q(k) \ = \ \frac{\bar{s}}{k^\gamma}\prod_{j=1}^k \frac{j^\gamma}{\bar{s} + j^\gamma}, \ \
{\rm and \  }\bar{s} {\rm \ is \ determined\ by \ \ } 1\ = \ \sum_{k=1}^\infty \prod_{j=1}^k
\frac{j^\gamma}{\bar{s} + j^\gamma}.\end{equation*}
Asymptotically, as $k\uparrow\infty$, when $0<\gamma<1$, $\log q(k) \sim -(\bar{s}/(1-\gamma))k^{1-\gamma}$ is in `stretched exponential' form; when $\gamma<0$, 
$\log q(k)\sim \gamma k \log k$; when $\gamma=0$, the case of uniform attachment when an old vertex is
selected uniformly, $\bar{s}=1$ and $q$ is geometric: $q(k) = 2^{-k}$ for
$k\geq 1$.

\medskip
{\it Superlinear case.}
 Finally, when $\gamma>1$, `explosion' or a sort of
`condensation' happens in that in the limiting graph a random
single vertex dominates in accumulating connections.  Let $\A = \lfloor \gamma/(\gamma-1)\rfloor$ and
let $\S_j$ be the collection of rooted trees with $j$ or less nodes. 

In \cite{Oliveira},
the
limiting graph is shown to be a tree with the following structure a.s.
\begin{eqnarray}
\label{superlinear}
&&{\rm  There\  is\  a\  single\  (random)\  vertex\  } v \ {\rm with\  an\  infinite\
number\  of\  children.} \nonumber \\
&&{\rm To \ this \ vertex,} \ v \ {\rm \ infinite\ copies \ of }  \ \S \ {\rm are \ glued \  for \ each \ } \S\in \S_\A.\\
&&  {\rm The \ remaining \ nodes \ in \ the \ tree \ form \ a \ finite \ collection \ }\T\nonumber \\
&&{\rm of \ bounded \ (but \ arbitrary) \ degree \ vertices.} \nonumber
\end{eqnarray}
We remark, by the definition of $\S_\A$, in the limit tree, there are an infinite number of nodes of degree $k$ for $1\leq k\leq \A$.

One may ask how the `condensation' effect arises in the dynamics with respect to the degree distribution.   In \cite{KRL} (see also \cite{KR}), non-rigorous rate formulation derivations of the mean orders of growth of $\{Z_k(n)\}$ give that
\begin{eqnarray}
\label{KR_pic}
\lim_{n\uparrow \infty} \frac{1}{n^{k - (k-1)\gamma}}E[Z_k(n)] & = &a_k \ \ \ {\rm when \ }1\leq k< \frac{\gamma}{\gamma -1}
\nonumber\\  
\lim_{n\uparrow\infty}\frac{1}{\log(n)}E[Z_k(n)]&=& b_k \ \ \ {\rm when \ }2\leq k=\frac{\gamma}{\gamma -1} \ {\rm is \ an \ integer.}\nonumber\\
\lim_{n\uparrow\infty} E[\Phi_k(n)] &<&\infty \ \ \ {\rm when \ }k> \frac{\gamma}{\gamma -1}
\end{eqnarray}
where
\begin{equation}\label{a_k}
 a_k = \prod_{j=2}^k \frac{w(j-1)}{j-(j-1)\gamma}\ \  {\rm and \ } \ b_k = w(k-1)a_{k-1}.
 \end{equation}

 See Section 4.2 p.~92-94 in \cite{Durrettbook}, which discusses \cite{KRL}, and the roles that the ansatz  and the limit,
\begin{equation}
\label{durrett_question}
E\Big[\frac{Z_k(n)}{S(n)}\Big] \sim \frac{E[Z_k(n)]}{E[S(n)]} \ \ {\rm  and\ \ } \lim_{n\uparrow\infty}\frac{E[S(n)]}{n^\gamma} = 1,
\end{equation}
assumed in \cite{KRL}, play in the approximations used to obtain \eqref{KR_pic}.

We remark, however, in \cite{Athreya}, that an a.s. form of the limit for $k=1$, that is $\lim_{n\uparrow\infty}\frac{1}{n}Z_1(n) = 1$ a.s.,  was proved through branching process methods.  Also, we observe, from \eqref{superlinear}, that one can deduce that $\lim_{n\uparrow\infty}\Phi_k(n)<\infty$ a.s. for $k>\gamma/(\gamma-1)$.

The stratification of growth orders, or `connectivity transitions', appears to be a phenomenon in a class of superlinear attachment schemes.  See \cite{KR} and \cite{KRL} for more discussion.  In this context, we remark the work \cite{CHJ} considers a preferential attachment urn model, different from the above graph model, and gives stratification results of the growth of mean orders of the counts of variously sized urns, and of the maximum sized urn (cf.  Lem. 3.4, Thm. 3.5, Cor. 3.6 in \cite{CHJ}).

In this context, the purpose of this note is to give, for the graph superlinear model, when $\gamma>1$, a rigorous, self-contained derivation of a.s. versions of \eqref{KR_pic} in Theorem \ref{degree_thm}, and to elaborate a more general asymptotic picture, including lower order terms and fluctuations, in Theorem \ref{mainthm2}.

 These derivations are based on rate equations and martingale analysis, with respect to the dynamics of the counts.  We were inspired by the interesting works \cite{CHJ}, \cite{KRL} and \cite{Oliveira}.  Although our arguments follow part of the outline in \cite{CHJ}, they do not make use of \eqref{superlinear} or its branching process/tree constructions in \cite{Oliveira}, or the more combinatorial estimates in \cite{CHJ}.  
     We also remark that the methods given here seem robust and may be of use in other superlinear preferential schemes, beyond the `standard' graph model that we have concentrated upon.

In the next section, we present our results and give their proofs in Section \ref{proofs}.
\subsection{Results}

Recall the formulas for $\{a_k\}$ and $\{b_k\}$ in \eqref{a_k}.

\begin{theorem}
\label{degree_thm}  The following structure for the degree sequence holds.

\noindent (1) We have
\begin{itemize}
\item[(a)] When $\frac{\gamma}{\gamma -1}>k\geq 1$,
$$\lim_{n\uparrow\infty}  \frac{1}{n^{k - (k-1)\gamma}}Z_k(n) \ = \ a_k \ \ {\rm a.s.}.$$
\item[(b)] When $k=\frac{\gamma}{\gamma -1}\geq 2$ is an integer, 
$$\lim_{n\uparrow\infty} \frac{1}{\log(n)} Z_k(n) \ = \ b_k \ \ {\rm a.s.}.$$
\item[(c)]
When $k>\frac{\gamma}{\gamma -1}$,
$$\lim_{n\uparrow\infty} \Phi_k(n) \ <\ \infty \ \ {\rm a.s.}$$
\end{itemize}
\noindent(2) Also, a.s., for $n\geq n_0$, where is $n_0$ is a random time, $V(n)$ is achieved at a fixed (random) vertex, and
$$\lim_{n\uparrow \infty} \frac{1}{n}V(n) = \ 1.$$ 
Moreover, all other vertices, in the process $\{G_n\}$, are of bounded degree.

\end{theorem}

The previous theorem represents a different, more quantitative, but also less descriptive, in terms of connectivity, version of the quite detailed structure in \eqref{superlinear}.

We now elaborate further upon these growth orders.

\begin{theorem} 
\label{mainthm2}
 Suppose
$\frac{\gamma}{\gamma -1}>k \geq 2$, and let $k^* = \lfloor \frac{\gamma}{2(\gamma-1)} + \frac{k}{2}\rfloor$ denote the largest integer such that $k^* - (k^*-1)\gamma \geq \frac{1}{2}[k - (k-1) \gamma]$.  Then,  with respect to coefficients $c^k_k=a_k, c^k_{k+1}, \ldots, c^k_{k^*}$, which can be found from equations \eqref{c_decomp} (see also section \ref{sec:algebra}), we have
$$
\frac{1}{n^{(k - (k-1)\gamma)/2}} \Big\{Z_k(n) - c^k_k n^{k+\gamma - k\gamma} -\cdots - c^k_{k^*} n^{k^* - (k^*-1)\gamma}\Big\} \Rightarrow N(0, a_k).
$$
Also,  when $k=\frac{\gamma}{\gamma -1}\geq 2$ is an integer, we have, for $b_k = w(k-1) a_{k-1}$, 
$$\frac{1}{(\log(n))^{1/2}} \Big\{Z_k(n) - b_k\log(n)\Big\} \Rightarrow N(0, b_k).$$
In addition, when $\frac{\gamma}{\gamma -1}<2$, we have
$$
\lim_{n \uparrow \infty} \Big(Z_1(n) - n \Big)  <  \infty \ \ \mathrm{a.s.}\ \ \mathrm{ and } \\
\lim_{n \uparrow \infty} \Big(V(n) - n \Big)  <  \infty \ \ \mathrm{a.s.}
$$
\end{theorem}

We remark that the last claim, when $\frac{\gamma}{\gamma -1}<2$, also follows from the description \eqref{superlinear}, although we will give here a different proof.

\begin{cor} 
\label{mainthm3}
Suppose $\frac{\gamma}{\gamma -1} > 2$ and recall $2^* = \lfloor \frac{\gamma}{2(\gamma-1)} \rfloor+1$. With respect to coefficients $c^{1}_2,c^1_{3}, \ldots, c^1_{2^*}$, and $m_2,m_3, \ldots,m_{2^*}$  given by equations \eqref{b_decomp} and coefficients $c^2_2=a_2, c^2_{3}, \ldots, c^2_{2^*}$, obtained from equations \eqref{c_decomp}, we have
$$
 \frac{1}{n^{1-\gamma/2}} \begin{Bmatrix} \begin{pmatrix}Z_1(n) - n \\ -Z_2(n) \\ V(n)-n \end{pmatrix} +\displaystyle{ \sum_{\ell=2}^{2^*}} \begin{pmatrix} c^1_\ell \\ c^2_{\ell} \\ m_{\ell}\end{pmatrix}  n^{\ell - (\ell-1)\gamma} 
\end{Bmatrix} \Rightarrow N(\mathbf{0}, a_2 \mathbbm{1}) 
$$
where $\mathbbm{1}$ is the $3 \times 3$ matrix with all the entries equal to 1.

Also, when $\frac{\gamma}{\gamma-1} =2$ (that is $\gamma =2$), we have
$$
 \frac{1}{\sqrt{\log(n)}} \begin{Bmatrix} \begin{pmatrix}Z_1(n) - n \\ -Z_2(n) \\ V(n)-n \end{pmatrix} +  \begin{pmatrix} b_2 \\ b_2 \\ b_2\end{pmatrix}  \log(n) 
\end{Bmatrix} \Rightarrow N(\mathbf{0}, b_2 \mathbbm{1}). 
$$

\end{cor} 

We remark that the weak convergence statements in Theorem~\ref{mainthm2}~and~Corollary~\ref{mainthm3} appear to be the first fluctuation results for nonlinear preferential attachment schemes.  We note, however, for the linear preferential model, a central limit theorem for the counts has been shown in \cite{Mori-01}. 

Corollary~\ref{mainthm3} shows that the  fluctuations for $Z_1,Z_2$ and $V$ are of the same order, and (in the limit) perfectly either correlated or anti-correlated. As seen in the proof, this `leading order balance' is a reflection of the structure of the two linear, deterministic, `constraints' \eqref{conservation} that are satisfied by the random counts $Z_k(n)$ in every realization.

\section{Proofs}
\label{proofs}
The proof section is organized as follows.  After development of a martingale framework in section~\ref{sec:martingale}, and combination of some stochastic analytic estimates, we prove Theorem \ref{degree_thm} at the end of section~\ref{sec:coarse}. Also, we prove Theorem \ref{mainthm2} and Corollary \ref{mainthm3}  in section~\ref{sec:fine}.

\subsection{Martingale decompositions}
\label{sec:martingale}
Define, for $j\geq 1$, the increment $d_k(j+1) = Z_k(j+1) - Z_k(j)$.   Given ${\mathcal F_j} = \sigma\big\{G_1, \ldots, G_j\big\}$, for $k\geq 2$, $d_{k}(j+1)$ has distribution
$$
d_k(j+1) \ = \ \left\{\begin{array}{rl}
1 & \ {\rm with \ prob.} \ \frac{w(k-1)Z_{k-1}(j)}{S(j)}\\
-1& \ {\rm with \ prob.} \ \frac{w(k)Z_k(j)}{S(j)}\\
0& \ {\rm with \  prob.}\ 1- \frac{w(k-1)Z_{k-1}(j)}{S(j)}- \frac{w(k)Z_k(j)}{S(j)}
\end{array}\right.
$$
When $k=1$,
$$d_1(j+1) \ = \ \left\{\begin{array}{rl}
1& \ {\rm with \ prob. } \ 1-\frac{w(1)Z_1(j)}{S(j)}\\
0& \ {\rm with \ prob.} \frac{w(1)Z_1(j)}{S(j)}.
\end{array}\right.
$$

Recall the count of the number of vertices with degree at least a certain level $k\geq 1$, 
$$\Phi_{k}(j) \ =\  \sum_{\ell\geq k}Z_\ell(j).$$ 
Let $e_{k}(j+1) = \Phi_{k}(j+1) - \Phi_{k}(j)$ be the corresponding increment.  Conditional on ${\mathcal F_j}$, for $k\geq 2$,
$$e_{k}(j+1) \ = \ \left\{\begin{array}{rl}
1 & \ {\rm with \ prob.} \ \frac{w(k-1)Z_{k-1}(j)}{S(j)}\\
0& \ {\rm with \ prob.} \ 1-\frac{w(k-1)Z_{k-1}(j)}{S(j)}\end{array}\right.$$

We note $\Phi_k(j)$ is an increasing function of $j$. Of course, $\Phi_1(j) = j+1$ as the number of vertices in $G_j$ equals $j+1$. 
The formula, for $k\geq 1$,
$$Z_k(n) \ = \ \Phi_k(n) - \Phi_{k+1}(n),$$
 relates $\{Z_k(n)\}$ to $\{\Phi_k(n)\}$.
 
 Finally, recall $V(j) = \max\{k | Z_k(j) \geq 1\}$ is the maximum over the degrees of all the vertices in the graph $G_n$.  Define its increment $h(j+1) = V(j+1)-V(j)$. 
 Conditional on ${\mathcal F}_j$,
$$h(j+1) \ = \ \left\{\begin{array}{rl}
1 & \ {\rm with \ prob.} \ \frac{w(V(j))Z_{V(j)}(j)}{S(j)}\\
0& \ {\rm with \ prob.} \ 1-\frac{w(V(j))Z_{V(j)}(j)}{S(j)}\end{array}\right.$$

Define the mean-zero martingales, with respect to $\{{\mathcal F_j}\}$,
\begin{eqnarray*}
M_k(n) &=& \sum_{j=1}^{n-1} d_k(j+1) - E\big[d_k(j+1)|{\mathcal F_j}\big]\\
Q_{k}(n) & = & \sum_{j=1}^{n-1}e_{k}(j+1) - E\big[e_{k}(j+1)|{\mathcal F_j}\big] \\
R(n) & = & \sum_{j=1}^{n-1}h(j+1) - E\big[h(j+1)|{\mathcal F_j}\big].
\end{eqnarray*}

The conditional expectations are computed as follows.
$$
 E\big[d_k(j+1)|{\mathcal F_j}\big] \ = \ \left\{\begin{array}{rl}
\frac{w(k-1)Z_{k-1}(j)}{S(j)} - \frac{w(k)Z_{k}(j)}{S(j)} & \ {\rm for \ }k\geq 2\\
1 - \frac{w(1)Z_1(j)}{S(j)} & \ {\rm for \ }k=1.
\end{array}\right.
$$
Also, for $k\geq 2$,
$$ E\big[e_k(j+1)|{\mathcal F_j}\big] \ = \ 
\frac{w(k-1)Z_{k-1}(j)}{S(j)}.$$
This immediately yields
$$ E\big[h(j+1)|{\mathcal F_j}\big] \ = \ 
\frac{w(V(j))Z_{V(j)}(j)}{S(j)}.$$

\medskip
Putting together terms, we have the following martingale decompositions:  
For $k=1$, 
\begin{equation}
\label{mart_decomp}
Z_1(n) - Z_1(1) \ = \ \sum_{j=1}^{n-1} \Big(1 - \frac{w(1)Z_1(j)}{S(j)}\Big) + M_1(n)
\end{equation}
and
for $k\geq 2$,
\begin{align*}
Z_k(n) - Z_k(1) & =  \sum_{j=1}^{n-1} E\big[d_k(j+1)|{\mathcal F_j}\big] + M_k(n)  \\
&= \sum_{j=1}^{n-1} \Big(\frac{w(k-1)Z_{k-1}(j)}{S(j)} - \frac{w(k)Z_k(j)}{S(j)}\Big) + M_k(n).\nonumber 
\end{align*} 
Also, for $k\geq 2$,
\begin{align*}
\Phi_{k}(n) - \Phi_{k}(1) 
&= \sum_{j=1}^{n-1} \frac{w(k-1)Z_{k-1}(j)}{S(j)} + Q_k(n). 
\end{align*}
Finally, for the maximum degree $V(n)$,
\begin{equation*}
V(n)-V(1)  \ = \ \sum_{j=1}^{n-1} \frac{w(V(j))Z_{V(j)}(j)}{S(j)} + R(n)
\end{equation*}
In the above equations, since we start from initial graph $G_1$, composed of two vertices attached by an edge, $Z_1(1) = 2$, $Z_k(1)=0$ for $k\geq 2$, and $V(1) = 1$.  
Also, $\Phi_k(1) =0$ for $k\geq 2$.

We will want to estimate the order of the martingales.  Their (predictable) quadratic variations can be computed:
\begin{eqnarray*}
\big\langle M_k\big\rangle(n) &= & \sum_{j=1}^{n-1} E\Big[\big(d_k(j+1) - E[d_k(j+1)|{\mathcal F_j}]\big)^2|\mathcal{F}_j\Big]\\
\big\langle Q_{k}\big\rangle(n)  &= & \sum_{j=1}^{n-1} E\Big[\big(e_{k}(j+1) - E[e_{k}(j+1)|{\mathcal F_j}]\big)^2|\mathcal{F}_j\Big] \\
\big\langle R\big\rangle(n)  &= & \sum_{j=1}^{n-1} E\Big[\big(h(j+1) - E[h(j+1)|{\mathcal F_j}]\big)^2|\mathcal{F}_j\Big].
\end{eqnarray*}
These martingales have bounded increments $|d_k|, |e_k|, |h|
 \leq 1$, and so 
 $
\langle M_k\rangle(n)$, $\langle Q_{k}\rangle(n)$, $\langle R\rangle(n)  
\leq  n$.
However, we will need better estimates in the sequel.  We have, when $k=1$,
\begin{equation}
\label{quad_comp}
E\big[\big(d_1(j+1) - E[d_1(j+1)|{\mathcal F_j}]\big)^2|\mathcal{F}_j\big] =
 \frac{w(1)Z_1(j)}{S(j)} \ \Big(1 - \frac{w(1)Z_1(j)}{S(j)}\Big) 
\end{equation}
and, when $k\geq 2$, that
\begin{eqnarray*}
&&E\big[\big(d_k(j+1) - E[d_k(j+1)|{\mathcal F_j}]\big)^2|\mathcal{F}_j \big]\\
&&\  = \Big(1 - \Big(\frac{w(k-1)Z_{k-1}(j)}{S(j)} - \frac{w(k)Z_{k}(j)}{S(j)}\Big)\Big)^2\frac{w(k-1)Z_{k-1}(j)}{S(j)}\\
&&\  \ + \Big(-1 - \Big(\frac{w(k-1)Z_{k-1}(j)}{S(j)} - \frac{w(k)Z_{k}(j)}{S(j)}\Big)\Big)^2\frac{w(k)Z_{k}(j)}{S(j)}\\
&&\  \ + \Big(\frac{w(k-1)Z_{k-1}(j)}{S(j)} - \frac{w(k)Z_{k}(j)}{S(j)}\Big)^2\Big(1 - \frac{w(k-1)Z_{k-1}(j)}{S(j)} + \frac{w(k)Z_{k}(j)}{S(j)}\Big).
\end{eqnarray*}
With respect to $Q_k(n)$,
for $k\geq 2$, we have
\begin{align}
\label{var_eq_phi}
&E\big[(e_{k}(j+1) - E[e_{k}(j+1)|{\mathcal F_j}])^2|\mathcal{F}_j\big] \\
& \  \ =  \Big(1-\frac{w(k-1)Z_{k-1}(j)}{S(j)}\Big)^2\frac{w(k-1)Z_{k-1}(j)}{S(j)} \nonumber \\
& \ \ \ \ \ \ \ \ \ \ \ \ \ \ + \Big(\frac{w(k-1)Z_{k-1}(j)}{S(j)}\Big)^2\Big(1-  \frac{w(k-1)Z_{k-1}(j)}{S(j)}\Big). \nonumber
\end{align}
For $R(n)$, we can compute
\begin{align*}
&E\big[(h(j+1) - E[h(j+1)|{\mathcal F_j}])^2|\mathcal{F}_j \big] \\
& \  \ =  \Big(1-\frac{w(V(j))Z_{V(j)}(j)}{S(j)}\Big)\frac{w(V(j))Z_{V(j)}(j)}{S(j)}. \nonumber 
\end{align*}

For later reference, we recall a law of the iterated logarithm, discussed in \cite{Fisher}, for martingales $M(n)$, say with bounded increments, and whose predictable quadratic variation diverges a.s. as $n\rightarrow\infty$ (cf. Remark 3 in \cite{Fisher}):  Almost surely, 
\begin{equation}
\label{fisher}
0 \ < \  \limsup_{n\rightarrow\infty} \frac{|M(n)|}{(2\langle M\rangle(n)\log\log(\langle M\rangle(n)\vee e^2))^{1/2}} \ < \ \infty.
\end{equation}

Also, a basic convergence result that we will use often is the following (cf. Thm. 5.3.1 in \cite{Durrett_prob}):  For a bounded difference martingale $M_n$, with respect to a probability $1$ set of realizations, 
\begin{align}
\label{bounded_diff}
&\text{either } M_n {\rm \ converges \ } \\
&\ \ \ \text{or  both \ }\limsup_{n\uparrow\infty}M_n = \infty \ \text { and  } \liminf_{n\uparrow\infty}M_n = -\infty. \nonumber
\end{align}

\subsection{Leading order asymptotics of $Z_k(n),S(n)$ and $V(n)$.}
\label{sec:coarse}

In this section we prove Theorem~\ref{degree_thm} after a sequence of subsidiary results.

Recall that the total degree of the graph at time $n$ is $\sum_{k=1}^{n}kZ_k(n)=2n$ and $Z_k(n) = 0$ for $k > n$.  One can bound 
\begin{equation}
\label{durrett_comment}
S(n) \ = \ \sum_{k=1}^{n} w(k)Z_k(n)\leq n^{\gamma -1}\sum_{k=1}^nkZ_k(n) \ = \ 2n^\gamma.
\end{equation}
A sharper bound obtains from the following argument. From the conservation laws~\eqref{conservation}, we get  $\sum_{k} (k-1) Z_k(n) = n-1$. For $\alpha > 1$ and $k\geq 1$, 
we have $(k-1)^\alpha \geq k^\alpha - \alpha k^{\alpha-1}$.  Using this, along with $Z_k(n) = 0$ for $k > V(n)$, we arrive at  
\begin{align}
 \sum_{k=1}^{n} k^\alpha Z_k(n) & \leq \sum_{k=1}^{n} (k-1)^\alpha Z_k(n) + \alpha \sum_{k=1}^{n} k^{\alpha-1} Z_k(n) \nonumber \\
& \leq (V(n)-1)^{\alpha-1} \sum_{k=2}^{V(n)}(k-1)Z_k(n) + \alpha \sum_{k=1}^{V(n)} k^{\alpha-1} Z_k(n) \nonumber \\
& = (n-1)(V(n)-1)^{\alpha-1} + \alpha \sum_{k=1}^{V(n)} k^{\alpha-1} Z_k(n) 
\label{recursion}
\end{align}
For $1 < \alpha \leq 2$, we can use $k^{\alpha-1} \leq k, \sum k Z_k = 2n$, to get 
\begin{equation}
\sum k^\alpha Z_k(n) \leq (n-1)(V(n)-1)^{\alpha-1} + 2 n \alpha \leq n(V(n)^{\alpha-1} + C_\alpha),
\label{alpha_lesseq_2}
\end{equation}
where $C_\alpha = 2 \alpha$ is a constant that only depends on $\alpha$. If $p = \lceil \alpha -1 \rceil \geq 2 $ so that $2 \leq p < \alpha \leq p+1$,  we iterate estimate~\eqref{recursion}  $p$ times to get
\begin{align}
\sum k^\alpha Z_k(n) \ &\leq   (n-1)(V(n)-1)^{\alpha-1} + (n-1)\alpha (V(n)-1)^{\alpha-2} + \cdots \nonumber \\
& + (n-1)  \alpha (\alpha-1)\cdots (\alpha-p+2) (V(n)-1)^{\alpha-p}  \nonumber \\
 & + \alpha(\alpha-1) \cdots (\alpha -p+1) \sum k^{\alpha-p} Z_k(n) \nonumber \\
  & \leq n (V(n) + C_\alpha)^{\alpha-1}
  \label{alpha_gtr_2}
\end{align}
where we use 
$0 < \alpha - p \leq 1$ and $k^{\alpha-p} \leq k$,  noting $\sum k Z_k(n) = 2n$. We obtain the final estimate using the Binomial expansion (with remainder), noting that $V(n) \geq 1, \alpha > p \geq  2$ and $C> 0$ together imply, for some $\theta \in (0,1)$,
\begin{align*}
(V(n) + C)^{\alpha-1}  & = V(n)^{\alpha-1} + (\alpha-1) V(n)^{\alpha-2} C + \cdots \\
& \ \ \ \ + \frac{(\alpha-1)(\alpha-2) \cdots (\alpha -p+1)}{(p-1)!}V(n)^{\alpha-p} C^{p-1} \\
& \ \ \ \ + \frac{(\alpha-1)(\alpha-2) \cdots (\alpha -p)}{p! \cdot (V(n)+ \theta C)^{p+1-\alpha}}  C^p \\
& \geq V(n)^{\alpha-1} + (\alpha-1)V(n)^{\alpha-2}C  \\
& \ \ \ \ + \sum_{\ell=2}^{p-1} \frac{(\alpha-1)(\alpha-2) \cdots (\alpha -\ell)}{\ell!}V(n)^{\alpha-\ell-1} C^\ell.
\end{align*} 
We thus obtain~\eqref{alpha_gtr_2} by  noting $V(n)\geq 1$ and picking $C_\alpha$ sufficiently large, so that
\begin{align*}
(\alpha-1)C_\alpha & \geq \alpha +  2 \alpha(\alpha-1) \cdots (\alpha -p+1)\\
C_\alpha^\ell & \geq \frac{\alpha \cdot \ell!}{\alpha-\ell}, \qquad \ell = 2,3,\ldots,p-1, 
\end{align*}  

Applying the estimates~\eqref{alpha_lesseq_2}~and~\eqref{alpha_gtr_2} to the weight $S(n)$, we get
\begin{equation}
\label{boundinmax}
S(n)  = \sum_{k=1}^n k^\gamma Z_k(n) \leq \begin{cases} n(V(n)^{\gamma-1} + C_\gamma) & {\rm for \ }1 < \gamma \leq 2 \\ n (V(n)+C_\gamma)^{\gamma -1} & {\rm for \ }\gamma > 2 \end{cases}
\end{equation}
Clearly, $V(n) \leq n$ so that $\limsup_n S(n) / n^\gamma \leq 1$.
This estimate is better than the bound in~\eqref{durrett_comment} by a factor of 2. More importantly,  the bound in \eqref{boundinmax} involves $V(n)$ and is therefore useful for our purposes.  In passing, we note another bound of $S(n)$, namely $S(n)\leq n^\gamma + n$, is obtained in equation (4.2.3) in \cite{Durrettbook}, by an optimization procedure. 

\medskip 
In what follows, with respect to (possibly random) quantities $U(n)$ and $u(n) > 0$, we signify that 
\begin{itemize}
\item $U(n)$ is $O(u(n))$ if $\limsup_{n \uparrow \infty} |U(n)|/u(n) < \infty$ a.s.
\item $U(n)$ is $o(u(n))$ if $\limsup_{n \uparrow \infty} |U(n)|/u(n) = 0$ a.s.  
\item $U(n)\gtrsim u(n)$ if $\liminf U(n)/u(n)>0$ a.s.
\item $U(n)\lesssim u(n)$ if $\limsup U(n)/u(n) <\infty$ a.s.
\item $U(n)\approx u(n)$ if $U(n)\gtrsim u(n)$ and $U(n)\lesssim u(n)$.
\end{itemize}

\medskip

We will need a preliminary estimate on the growth of $V(n)$.
\begin{proposition} 
\label{max-degree}
For each $0<\epsilon<1$, the maximum degree $V(n)$ satisfies 
$$
\liminf_{n \to \infty} \frac{V(n)}{n^{1-\epsilon}} > 0 \mbox{ a.s.}
$$
\end{proposition}

\begin{proof}

{\it Step 1.}  First, we observe that $V(n) \nearrow \infty$ a.s.
The probability that the maximum value increases by $1$ at time $n+1$, given $\mathcal{F}_n$, is bounded below, $$\frac{Z_{V(n)}(n)V(n)^\gamma}{S(n)} \geq \frac{C}{n},$$
 using \eqref{boundinmax} and $Z_{V(n)}(n), V(n)\geq 1$, so that, respectively for $1 < \gamma \leq 2$ and $\gamma > 2$, 
 $$
 \left[\frac{V(n)^{\gamma-1}}{V(n)^{\gamma-1}+C_\gamma}\right] \geq (1+C_\gamma)^{-1}
 \text{ and } 
 \left[\frac{V(n)}{V(n)+C_\gamma}\right]^{\gamma-1}\geq (1+C_\gamma)^{1-\gamma} .
 $$
Hence, $\sum P(V(n+1)-V(n) = 1|\mathcal{F}_n)= \infty$, and it follows from a Borel-Cantelli lemma (Thm. 5.3.2 in \cite{Durrett_prob}) that $V(n) \nearrow \infty$ a.s.  We note a version of this argument is used for the model in \cite{CHJ}.

{\it Step 2.} The idea now is consider $V(n)$ in a $\log$ scale.  Consider the evolution of $\Delta(n) = \sum_{j=1}^{V(n)-1} j^{-1} \approx \log(V(n))$. As $V(n) \nearrow \infty$, we have that $\Delta(n) \nearrow \infty$ a.s.  Note also, our initial condition on $G_1$ yields $\Delta(2) = 1$. 

 The increments $\delta$ of $\Delta$ are given, conditioned on $\mathcal{F}_n$, by
\begin{equation*}
\delta(n+1)  = \Delta(n+1)-\Delta(n) =  \begin{cases}  V(n)^{-1} & \mbox{ with prob. } \frac{Z_{V(n)} V(n)^\gamma}{S(n)} \\
 0 & \mbox{ with prob. } 1 - \frac{Z_{V(n)} V(n)^\gamma}{S(n)} \\
\end{cases}
\end{equation*}
Let also $\rho$ denote the associated martingale, with differences $\delta(j+1)-E[\delta(j+1)|\mathcal{F}_j]$ (bounded by $1$), and quadratic variation 
$$\langle \rho\rangle(j)= \sum_{j=1}^{n-1} E\big[(\delta(j+1)-E[\delta(j+1)|\mathcal{F}_j])^2|\mathcal{F}_j].$$ 
We may compute
\begin{eqnarray}\label{quad_var_Delta}
&&E[(\delta(n+1) - E[\delta(n+1)\ | \ \mathcal{F}_n])^2|\mathcal{F}_n] \\
&&\ \ \ \ \  = \ \frac{1}{V(j)^2}\Big(1-\frac{Z_{V(j)}(j)V(j)^{\gamma}}{ S(j)}\Big)\frac{Z_{V(j)}(j)V(j)^{\gamma}}{S(j)} \leq \frac{1}{V(j)^{2}}. \nonumber
\end{eqnarray}

{\it Step 3.} From the martingale decomposition of $\Delta$, we get
\begin{equation*}
\Delta(n) = 1 + \sum_{j=2}^{n-1} \frac{Z_{V(j)}(j)V(j)^{\gamma-1}}{S(j)} + \rho(n).
\end{equation*}

From \eqref{boundinmax} and $Z_{V(n)}(n)\geq 1$, we obtain, for $n \geq 2$, that
\begin{align}
\Delta(n+1) - \Delta(n) & \geq \frac{V(n)^{\gamma-1}}{S(n)} + \rho(n+1) - \rho(n) \nonumber \\
& \geq \frac{D_\gamma(V(n))}{n}   + \rho(n+1) - \rho(n) , 
\label{delta-evolve}
\end{align}
where
\begin{equation*}
D_\gamma(V(n)) = \begin{cases}   \left(1 + C_\gamma/V(n)^{\gamma-1}\right)^{-1}  & {\rm for \ }1 < \gamma \leq 2, \\ 
\left(1 + C_\gamma/V(n)\right)^{1-\gamma} & {\rm for \ }\gamma > 2, \end{cases}
\end{equation*}
so that $D_\gamma(V(n)) \nearrow 1$ as $V(n) \nearrow \infty$ for all $\gamma > 1$.

{\it Step 4.}  
Since $V(n) \nearrow \infty$ a.s., by bounded convergence, $\lim_{n\uparrow\infty} E[D_\gamma(V(n))] = 1$. Consequently, given $0<\epsilon_0<1$, there is a sufficiently large deterministic time $n_0$ such that 
$ E[D_\gamma(V(n))] \geq 1-\epsilon_0$,
for $n \geq n_0$. Summing \eqref{delta-evolve} over times from $n_0$ to $n$ yields
$$E\Delta(n+1) \geq E\Delta(n_0) + \left(1-\epsilon_0\right)\sum_{j=n_0}^n\frac{1}{j}.$$
Then, $E\Delta(n) \gtrsim \left(1-\epsilon_0\right)\log(n)$.  By Jensen's inequality, and the bound $\log (V(n)-1) \geq  \Delta(n)-\Delta(2)$ a.s., following from $\int_1^b x^{-1}dx \geq \sum_{j=2}^{b}j^{-1}$, we have $\log EV(n) \geq E\log V(n)\gtrsim E\Delta(n)$ and so $EV(n) \gtrsim n^{1-\epsilon_0}$.

Hence, by Jensen's inequality, applied to $\phi(x)=x^{-2}$, and \eqref{quad_var_Delta}, we have
$$E\langle \rho\rangle(j+1) - E\langle \rho\rangle(j) \leq E\left[V(j)^{-2}\right] \leq \big(EV(j)\big)^{-2} \lesssim j^{-2(1-\epsilon_0)}.$$ 
Now note $E\langle \rho\rangle(n_0) \leq n_0$ by \eqref{quad_var_Delta} and $V(j)\geq 1$.  Then, with the choice $\epsilon_0 = 1/4$ say, monotone convergence yields 
$$
\lim E\langle \rho\rangle(n) = E\lim \langle \rho\rangle(n) \leq n_0 + C\sum_{j\geq n_0} j^{-2(1-\eps_0)} <\infty.
$$

{\it Step 5.} 
As $\sup_n E \rho_n^2 = \sup_n E\langle \rho\rangle(n) < \infty$ by Step 4, standard martingale convergence results (cf. Thm. 5.4.5 \cite{Durrett_prob}) imply that $\rho_n$ converges to a limit random variable $\rho_\infty$ a.s and also in $L^2$. We note, as $\rho_\infty \in L^2$, it is finite a.s.

{\it Step 6.} Given $0 < \epsilon < 1$, noting again $V(n)\nearrow\infty$ a.s., we may choose now a (random) time $n_1$ such that $D_\gamma(V(n)) \geq \left(1-\epsilon\right)$ and $\rho(n) \geq \rho_\infty -1$.
 Note also that $|\Delta(n) - \log(V(n))| < 1$ and $|\log(n) - \sum_{j=1}^{n-1} j^{-1}| < 1$ for all $n$.
Summing \eqref{delta-evolve}, from $n_1$ to $n$, we thus have 
$$\Delta(n) - \Delta(n_1) \geq \left(1-\epsilon\right)\sum_{j=n_1}^{n-1} \frac{1}{j} +\rho(n)-\rho(n_1)$$
and so
$\Delta(n) - (1-\epsilon) \log n \geq \Delta(n_1) - (1-\epsilon) \log(n_1)+ \rho_\infty -3$ is bounded below a.s.  Hence, for a constant $C$, depending on the realization, we have a.s. that $V(n) \geq Cn^{1-\epsilon}$ for all large $n$. 
\end{proof}

We now address the growth of $Z_1(n)$ and $\{\Phi_k(n)\}$. 
\begin{proposition}
\label{Z_Phi_asymptotics}
Let $\epsilon > 0$ be such that  $\gamma' \equiv \gamma(1-\epsilon) > 1$ and $\gamma'/(\gamma'-1)$ is not an integer. We have a.s. that 
$$\lim_{n\uparrow\infty} \frac{M_1(n)}{n} = 0, \ \ \lim_{n\uparrow\infty}\frac{Z_1(n)}{n}=  1 \\ \  \ {\rm and \ \ }\limsup_{n\uparrow\infty}\frac{\Phi_k(n)}{n^{k-(k-1)\gamma'} \vee 1} < \infty.$$
\end{proposition}

\begin{proof}
Proposition~\ref{max-degree} implies that, for any $0<\epsilon <1$,  $S(j) \geq V(j)^\gamma \gtrsim j^{(1-\epsilon) \gamma}$.  Write, noting the decomposition \eqref{mart_decomp}, that
\begin{equation}
\label{decomp_Z_1}
\frac{1}{n}Z_1(n) \ = \ \frac{1}{n}Z_1(1) + \frac{1}{n}\sum_{j=1}^{n-1} \Big(1- \frac{w(1)Z_1(j)}{S(j)}\Big) + \frac{1}{n}M_1(n).
\end{equation}
Since $Z_1(1)=2$, the first term on the right-hand side equals $2n^{-1}$. The second term can be estimated using $Z_1(j)\leq j+1$, so that
$$\frac{Z_1(j)}{S(j)} \ \leq \ \frac{j+1}{j^{\gamma'}} \ = \ O(j^{1-\gamma'}).$$
Then, if $1<\gamma'< 2$, the sum
\begin{equation}
\label{sum}
\sum_{j=1}^{n-1}\frac{Z_1(j)}{S(j)} \ = \ O(n^{2-\gamma'}).\end{equation}
If $\gamma' = 2$, this sum is $O(\log(n))$ and, if $\gamma'>2$, this sum is convergent.  
Note, in particular, in all cases that the sum in \eqref{sum} is $o(n)$ a.s.  

We now analyze the quadratic variation $\langle M_1 \rangle(n)$.  Noting \eqref{quad_comp}, we have $n+1-Z_1(n) = \sum_{j=1}^{n-1} \frac{w(1)Z_1(j)}{S(j)} + M_1(n)$ and 
$$\frac{1}{2}\sum_{j=j_0}^n \frac{w(1)Z_1(j)}{S(j)}\leq \langle M_1\rangle(n) \leq \sum_{j=1}^n \frac{w(1)Z_1(j)}{S(j)} = o(n),$$
where $j_0$ is chosen large enough.

  If the sum $\sum_{j=1}^{n-1} \frac{w(1)Z_1(j)}{S(j)}$ converges, then $M_1(n)$, being a bounded difference martingale, must converge, as otherwise $\liminf M_1(n) =-\infty$, which is impossible as $n-Z_1(n)\geq 0$ (cf. \eqref{bounded_diff}).  On the other hand, if the sum diverges, then, by the law of the iterated logarithm \eqref{fisher}, $M_1(n)$ is at most $(o(n)\log\log(o(n)))^{1/2}$.  In either case, $n^{-1}M_1(n)\rightarrow 0$ a.s.
Also, from the decomposition \eqref{decomp_Z_1},
we conclude that $Z_1(n)/n \rightarrow 1$ a.s.

We now turn to the counts $\{\Phi_\ell\}$.  Since $\Phi_1(j) = j+1$, we have $\lim_{n\uparrow\infty}\frac{\Phi_1(n)}{n} = 1$. Assuming that the proposition statement for $\{\Phi_\ell\}$ holds for index $k-1$, we will prove it for $k$. 

Note, by the induction hypothesis, 
$$
\sum_{j=1}^{n-1} \frac{w(k-1)Z_{k-1}(j)}{S(j)} \leq \sum_{j=1}^{n-1} \frac{w(k-1)\Phi_{k-1}(j)}{j^{\gamma'}} = O\big(n^{k - (k-1)\gamma'} \vee 1\big).
$$
By the computations \eqref{var_eq_phi}, the quadratic variation $\langle Q_k\rangle(n)$ satisfies
$$\frac{1}{2}\sum_{j=j_0}^{n-1} \frac{w(k-1)Z_{k-1}(j)}{S(j)}\leq \langle Q_k\rangle(n) \leq \sum_{j=1}^{n-1} \frac{w(k-1)Z_{k-1}(j)}{S(j)}  = O\big(n^{k - (k-1)\gamma'}\vee 1\big),
$$
where again $j_0$ is chosen large enough.  

Consider the equation 
$$\Phi_k(n) = \Phi_k(1) +\sum_{j=1}^{n-1} \frac{w(k-1)Z_{k-1}(j)}{S(j)}+ Q_k(n).$$
  If the sum $\sum_{j=1}^{n-1} \frac{w(k-1)Z_{k-1}(j)}{S(j)}$ converges, then as $\Phi_k$ is nonnegative, $\liminf Q_k(n)$ cannot equal $-\infty$; hence, $Q_k(n)$, being a bounded difference martingale, must also converge (cf.~\eqref{bounded_diff}).  If, however, the sum diverges, by the law of the iterated logarithm \eqref{fisher}, $Q_k(n)$ is of the order 
$$Q_k(n) = O\Big( \big(n^{k - (k-1)\gamma'}\log\log n^{k - (k-1)\gamma'}\big)^{1/2} \Big).$$  
In either case, we derive that $\Phi_k(n)/n^{k - (k-1)\gamma'}\vee 1 = O(1)$.
\end{proof}

Recall, for each $k\geq 1$, that $\Phi_k(n)$ is nondecreasing.  
By the last proposition, for $0<\epsilon<1$ such that $\gamma(1-\epsilon)>1$, there is a finite index $B$, in particular $B= \lceil \gamma(1-\epsilon)/(\gamma(1-\epsilon) -1\rceil$, such that $\lim_{n\uparrow\infty}\Phi_B(n)=:\Phi_B(\infty)$ is a.s. finite.  Moreover, there is a (random) finite time $N_B$ such that $\Phi_B(n) = \Phi_B(\infty)$ for $n\geq N_B$.
\begin{proposition}
\label{frozen}
Only one (random) node represented in the count $\Phi_B(\infty)$ has diverging degree, while all the other nodes in this count, and elsewhere in $\{G_n\}_{n\geq N_B}$, have bounded degree a.s.  
\end{proposition}

\begin{proof} Since the maximum degree $V(n)$ diverges a.s. (Proposition \ref{max-degree}), we have $\Phi_B(N_B)\geq 1$, and also that vertices from the set corresponding to $\Phi_B(N_B)$ are selected infinitely many times.     
Denote the (random) finite number of vertices corresponding to $\Phi_B(N_B)$ by $\{x_1,\ldots, x_m\}$.  There are countable possible configurations of this set.  

Now, given that an infinite number of selections are made in a fixed finite set $\{y_1,\ldots, y_m\}$ at times $\{s_j\}$, we see that vertex $y_i$ is chosen at time $s_{j+1}$, conditioned on the state at time $s_j$, with probability 
\begin{equation}\label{same_mc}
w\big(d_{y_i}(s_j)\big)/\sum_{\ell =1}^m w\big(d_{y_\ell}(s_j)\big).
\end{equation}
  Let $s_0 = 1$ and $\{y_1(j),\ldots, y_m(j\}_{j\geq 0}$ be the degrees of these vertices at times $\{s_j\}_{j\geq 0}$, according to this conditional update rule.  

Such a `discrete time weighted Polya urn', $\{y_1(j),\ldots, y_m(j)\}_{j\geq 0}$, can be embedded in a finite collection of independent continuous time weighted Polya urns, $\{U_1,\ldots, U_m\}$.  Let $U_i$ be an urn process which increments by one at time $t$ with rate $w\big(U_i(t)\big)$ and initial size given by $w(d_{y_i}(1))$.  Since $w(d)=d^\gamma$ and $\gamma>1$, so that $\sum w(d)^{-1}<\infty$, the urn $U_i$ explodes at a finite time $T_i$ a.s.  Let $\tau_j$ be the $j$th time one of the urns $\{U_i\}_{i=1}^m$ is incremented.  The discrete time weighted Polya urn system is recovered by evaluating the continuous time process at times $\tau_j$ for $j\geq 1$.  Indeed, from properties of exponential r.v.'s, given the state of the system at time $\tau_j$, the urn $U_i$ rings at time $\tau_{j+1}$ with probability $w\big(U_i(\tau_j)\big)/\sum_{\ell=1}^m   w\big(U_\ell(\tau_j)\big)$, the same as \eqref{same_mc}.

Now, one of the continuous time urns, say $U_i$, explodes first, meaning it is selected an infinite number of times.  Another urn $U_j$ cannot also explode at the same time $T_j=T_i$ since $T_i$, being a sum of independent exponential r.v.'s, is continuous and independent of $T_j$.  Hence, a.s. after a (random) time, selections in the finite system are made only from one (random) urn or vertex.  

Hence, by decomposing on the countable choices of the set $\{x_1,\ldots, x_m\}$ corresponding to $\Phi_B(N_B)$, we see that there is exactly one vertex in the set with diverging degree a.s.  All the others in the set are only selected a finite number of times, and therefore have finite degree a.s. 
 
 With respect to the graphs $\{G_n\}$, vertices in the counts $\{Z_k\}_{k=1}^{B-1}$, by definition have degree bounded by $B-1$.  Hence, all vertices in the graph process have bounded degree, except for the maximum degree vertex, which diverges a.s. Similar types of argument can be found in \cite{CHJ} and \cite{Athreya}. \end{proof}

We now derive a.s. and $L^1$ limits with respect to $V(n)$ and $S(n)$.  These we remark partly addresses the question in \eqref{durrett_question}.

\begin{proposition}
\label{S_asymptotics}
We have, a.s. and in $L^1$ that 
$$ \lim_{n\uparrow\infty}\frac{1}{n}V(n)\ = \ 1 \ \ {\rm and \ \ }
\lim_{n\uparrow\infty}\frac{1}{n^\gamma}S(n)\ = \ 1.$$
\end{proposition}

\begin{proof}
By Proposition \ref{Z_Phi_asymptotics}, $Z_1(n) = n + o(n)$ and $\Phi_2(n) = \sum_{k\geq 2}Z_k(n) = o(n)$ a.s.  By Proposition \ref{frozen}, $V(n)$ is achieved at a fixed (random) vertex, say $v$, a.s.  By \eqref{conservation} $\sum_{k\geq 1} kZ_k(n) = 2n$ so that, for all large $n$,
$$
V(n) = \Big(2 n - Z_1(n)\Big) - \sum_{k =2}^{V(n)-1} k Z_k(n) 
$$ 
Recall the index $B$ before Proposition \ref{frozen}, and let $\hat\Phi_B(\infty)$ be the vertices represented by $\Phi_B(\infty)$.  We have, by Propositions \ref{Z_Phi_asymptotics} and \ref{frozen}, for all large $n$, that
\begin{eqnarray*}
\sum_{k =2}^{V(n)-1} k Z_k(n) &=& \sum_{k=2}^{B-1} kZ_k(n) 
 + {\rm finite\ total\ degree \ of \ vertices \ in \ }\hat\Phi_B(\infty)\setminus\{v\} \\ 
& =&  o(n).
\end{eqnarray*}
Since $Z_1(n)/n \rightarrow 1$ a.s. by Proposition \ref{Z_Phi_asymptotics}, we obtain therefore that $V(n)/n\rightarrow 1$ a.s.

Further, for all large $n$, the weight of the graph equals
\begin{eqnarray*}
S(n) & = & {\rm weight \ of \ }v + Z_1(n) + \sum_{k=1}^{B-1} k^\gamma Z_k(n)\\
&&\ \ \ \  + {\rm finite \ total \ weight \ of \ vertices \ in \ }\hat\Phi_B(\infty)\setminus\{v\}\\
&  = & (n + o(n))^\gamma + n + o(n) \ \ {\rm  a.s.},
\end{eqnarray*} 
from which the desired a.s. limit for $S(n)/n^{\gamma}$ is recovered.

The $L^1$ limits follow from bounded convergence, as $0\leq V(n)\leq n$, and $0\leq S(n)/n^\gamma\leq 2$, say by \eqref{durrett_comment}.
\end{proof}

\medskip

We now show Theorem \ref{degree_thm}.

\medskip
\noindent{\bf Proof of Theorem \ref{degree_thm}.}  
To establish the first part, when $\frac{\gamma}{\gamma -1}>k\geq 1$, we show by induction
that a.s.
\begin{equation}
\label{mean_limit}
 \lim_{n\rightarrow \infty}\frac{1}{n^{k-(k-1)\gamma}}M_k(n) = 0 \ \ {\rm and \ \ } \lim_{n\rightarrow\infty} \frac{1}{n^{k - (k-1)\gamma}}Z_k(n) \ = \ a_k.
\end{equation}

The base case $k=1$ has been shown in Proposition \ref{Z_Phi_asymptotics}.   Suppose now that the claim holds for $M_{k-1}(n)$ and $Z_{k-1}(n)$.  
Write
$$\frac{1}{n^{k - (k-1)\gamma}}\Phi_k(n) \ = \ \frac{1}{n^{k-(k-1)\gamma}}\sum_{j=1}^{n-1} \frac{w(k-1)Z_{k-1}(j)}{S(j)} + Q_k(n).$$
By the asymptotics of $S(j)$ in Proposition \ref{S_asymptotics}, and the assumption, we have a.s.
$$\frac{Z_{k-1}(j)}{S(j)} \approx j^{k-1 - (k-1)\gamma}.$$
Then, as $k-(k-1)\gamma>0$, we have a.s.
$$\frac{1}{n^{k-(k-1)\gamma}}\sum_{j=1}^{n-1} \frac{w(k-1)Z_{k-1}(j)}{S(j)} \rightarrow \frac{w(k-1)a_{k-1}}{k - (k-1)\gamma}.$$
Also, by the computations \eqref{var_eq_phi}, the quadratic variation $\langle Q_k\rangle(n)$ satisfies a.s.
$$\langle Q_k\rangle(n) \approx \sum_{j=1}^{n-1} \frac{w(k-1)Z_{k-1}(j)}{S(j)} \approx \frac{w(k-1)a_{k-1}}{k-(k-1)\gamma}n^{k - (k-1)\gamma}.$$
Then, by the law of the iterated logarithm \eqref{fisher}, we have a.s.
$$ Q_k(n) = O\Big( \big(n^{k - (k-1)\gamma}\log\log n^{k - (k-1)\gamma}\big)^{1/2}\Big).$$  
Therefore, a.s., we obtain
$$
\lim_{n\rightarrow \infty}\frac{1}{n^{k-(k-1)\gamma}}Q_k(n) = 0 \ \ {\rm and \ \ } \lim_{n\rightarrow\infty} \frac{1}{n^{k - (k-1)\gamma}}\Phi_k(n) \ = \ \frac{w(k-1)a_{k-1}}{k - (k-1)\gamma}.$$

To update to $M_k(n)$ and $Z_k(n)$, we observe, by \eqref{quad_comp},
that the quadratic variation of $M_k(n)$ is of order a.s.
\begin{align}
\label{quadvar}
\langle M_k\rangle(n) &\lesssim \sum_{j=1}^{n-1} \frac{w(k-1)Z_{k-1}(j)}{S(j)} + \sum_{j=1}^{n-1} \frac{w(k)Z_k(j)}{S(j)}\\
&\lesssim C n^{k -(k-1)\gamma} +\sum_{j=1}^{n-1}\frac{\Phi_k(j)}{S(j)} \nonumber \\
&\lesssim Cn^{k-(k-1)\gamma} + C n^{k+1 - k\gamma}. \nonumber
\end{align}
Note that $k+1 - k\gamma <k - (k-1)\gamma$.  Hence, by the law of the iterated logarithm \eqref{fisher}, we have $M_k(n)/n^{k - (k-1)\gamma} \rightarrow 0$ a.s. Moreover,
\begin{eqnarray*}
Z_k(n) &=& \sum_{j=1}^{n-1} \frac{w(k-1)Z_{k-1}(j)}{S(j)} - \frac{w(k)Z_k(j)}{S(j)} + M_k(n)\\
&=&\sum_{j=1}^{n-1} \frac{w(k-1)Z_{k-1}(j)}{S(j)} + o(n^{k-(k-1)\gamma}).
\end{eqnarray*}
Therefore, we have a.s.
$$\lim_{n\rightarrow\infty} \frac{Z_k(n)}{n^{k - (k-1)\gamma}} = \frac{w(k-1)a_{k-1}}{k-(k-1)\gamma} = a_k,$$
concluding the proof of \eqref{mean_limit}.

To show the second item, when $\frac{\gamma}{\gamma -1} = k\geq 2$ is an integer, that is
a.s.,
\begin{equation*}
\label{log_limit}
\lim_{n\rightarrow\infty} \frac{1}{\log(n)}M_k(n) = 0 \ \ {\rm and \ \ } \lim_{n\rightarrow\infty}\frac{1}{\log(n)}Z_k(n) \ = \ b_k,
\end{equation*}
the argument is quite similar, following closely the steps above and is omitted. 

We now consider the third item, when $k>\gamma/(\gamma-1)$.  Let $r = \lfloor \gamma/(\gamma-1)\rfloor +1$ be the smallest integer strictly larger than $\gamma/(\gamma-1)$.  We now show that $\Phi_r(n)$ converges a.s.  Then, we would conclude, for $k>\gamma/(\gamma-1)$, as $\Phi_k(n)\leq \Phi_r(n)$ and $\Phi_k(n)$ is nondecreasing, that $\Phi_k(n)$ converges a.s.

Note $r-1\leq \gamma/(\gamma-1)$ and $\Phi_r(n) = \sum_{j=1}^{n-1} \frac{w(r-1)Z_{r-1}(j)}{S(j)} + Q_r(n)$.  In particular, by the previous items, $Z_{r-1}(j)/S(j)\approx j^{(r-1)(1-\gamma)}$ is summable as $(r-1)(\gamma-1)>1$.  Hence, $Q_r(n)$, being a bounded difference martingale, also must converge by \eqref{bounded_diff}, as its liminf is not $-\infty$ since $\Phi_r(n)\geq 0$.  Hence, $\Phi_r(n)$, being nondecreasing, converges a.s.  

Finally, the second part (2) follows from directly from Propositions \ref{frozen} and \ref{S_asymptotics}. 
\qed

\subsection{Refined asymptotics and fluctuations for $Z_k(n)$}
\label{sec:fine}

We now prove Theorem \ref{mainthm2} and Corollary \ref{mainthm3}, which give a description of the lower order nonrandom terms as well as the fluctuations in the counts $Z_k(n)$.
 
\medskip
\noindent
{\bf Proof of Theorem \ref{mainthm2}.}  
We consider first the part when $2\leq k< \gamma/(\gamma -1)$.  Recall that $A = \lfloor \gamma/(\gamma-1)\rfloor$, and the initial decomposition of $Z_k(n)$,
\begin{equation}
\label{Z_k_eq}
Z_k(n) = \sum_{j=1}^{n-1} \frac{w(k-1)Z_{k-1}(j)}{S(j)} - \frac{w(k)Z_k(j)}{S(j)} + M_k(n).
\end{equation}
By \eqref{conservation} and by Theorem \ref{degree_thm} that counts of vertices with degrees larger than $A$ are bounded, we have $n+1 = \sum_{k\geq 1} Z_k(n) = Z_1(n) + \cdots + Z_A(k) + O(1)$, and hence
\begin{equation*}
Z_1(n) = n - Z_2(n) - \cdots - Z_A(n) + O(1).
\end{equation*}
Also, by the remarks below \eqref{durrett_comment}, by Theorem \ref{degree_thm} again, and by Proposition \ref{frozen} that the maximum is achieved at a fixed vertex, we have, for all large $n$, that
$n-1=\sum_{k=2}^n (k-1)Z_k(n) = V(n) + \sum_{k=2}^A (k-1)Z_k(n) + O(1)$, and so $V(j) = j - \sum_{k=2}^A (k-1)Z_k(j) + O(1)$.  Hence, 
\begin{eqnarray*}
S(j) &= & V(j)^\gamma + \sum_{k=1}^A w(k)Z_k(j) + O(1)\\
&=& \big(j-\sum_{k=2}^A (k-1)Z_k(j)\big)^\gamma + \sum_{k=1}^A  w(k)Z_k(j) + O(1).
\end{eqnarray*}

One may now further decompose $Z_{k-1}(j)$, $Z_k(j)$ and $S(j)$, and successively decompose in turn the resulting counts.  Recall that, when $\gamma/(\gamma-1)$ is an integer, by Theorem \ref{degree_thm}, $Z_A(n) = O(\log(n))$.  Recall the leading orders of $\{Z_k(n)\}_{k=1}^A$ in Theorem \ref{degree_thm}.  Recall, also, the martingale $M_k(n)$, when $2\leq k<\gamma/(\gamma-1)$, is of order $\big (n^{(k - (k-1)\gamma}\log\log n^{k-(k-1)\gamma}\big)^{1/2}$, using \eqref{quadvar} and \eqref{fisher}.  Similarly, when $A=\gamma/(\gamma-1)$, the quadratic variation of $M_A(n)$ is of order $\log(n)$ and so $M_A(n)$ is of order $(\log(n)\log\log \log(n))^{1/2}$.

Any term, involving martingales, arising after the initial decomposition is of lesser order than $M_k(n)$ and may be omitted.  Indeed, the largest term arising from decomposing $S(j)$ in the sum in \eqref{Z_k_eq} is $\sum_{j=1}^{n-1} \frac{w(k-1)Z_{k-1}(j) (\gamma j^{-1}M_2(j))}{j^\gamma}$, which is of order
$n^{k - (k-1)\gamma - \gamma/2}\sqrt{\log\log n^{2-\gamma}}$.  On the other hand, the largest term arising from decomposing the numerators in the sum in \eqref{Z_k_eq} is
$\sum_{j=1}^{n-1}\frac{w(k-1)M_{k-1}(j)}{S(j)}$, which is of order 
$n^{(k-1 - (k-2)\gamma)/2 +1-\gamma}\sqrt{\log\log(n^{(k-1 - (k-2)\gamma)/2}} )$.  Given that $\gamma>1$, both orders are smaller than $n^{(k-(k-1)\gamma)/2}$.

Moreover, following these martingale calculations, any term $\eta(n)$ of order less than $n^{(k-1 - (k-2)\gamma)/2}$ in the decomposition of $Z_{k-1}(n)$ will contribute a term in the computation of $Z_k(n)$ of order less than $n^{(k - (k-1)\gamma)/2}$. 
Also, any term $\zeta_\ell(n)$ in the decompositions of $Z_\ell(n)$, with respect to $S(n)$, of order less than $n^{(2 - \gamma)/2}$ will give a term of order less than  $n^{(k - (k-1)\gamma)/2}$ in computing $Z_k(n)$.

After a finite number of iterative decompositions, we may decompose $Z_k(n)$ as the sum of a finite number of nonrandom diverging terms of orders larger or equal to $n^{(k-(k-1)\gamma)/2}$, the martingale $M_k(n)$, and a finite number of other (possibly random) terms whose order is less than $n^{(k - (k-1)\gamma)/2}$.

 By Theorem \ref{degree_thm}, as the dominant order of $Z_\ell(j)$ is $n^{\ell - (\ell-1)\gamma}$, when $\ell<\gamma/(\gamma-1)$, and $Z_A(j) = O(\log(j))$ when $A=\gamma/(\gamma-1)$, the nonrandom divergent orders, larger or equal to $n^{(k-(k-1)\gamma)/2}$, which appear in the decomposition of $Z_k(n)$ are all of the form $n^{m - (m-1)\gamma}$ where $k\leq m\leq k^*$, and $k^*$ is the largest integer so that $k^* - (k^*-1)\gamma \geq (k - (k-1)\gamma)/2$.   
 
By these decompositions, the coefficients $\{c^k_{k+\ell}\}_{\ell=0}^{k^*-k}$ are found from matching the powers of $n$ in the conditions
\begin{equation}\label{c_decomp}
Y_k(n) - \sum_{j=1}^{n-1} \frac{w(k-1)Y_{k-1}(j) -w(k)Y_k(j)}{T(j)} = o(n^{(k-(k-1)\gamma)/2})
\end{equation}
where, for $2\leq k<\gamma/(\gamma-1)$, 
$$Y_k(n) = c^k_kn^{k-(k-1)\gamma} + \cdots + c^k_{k^*} n^{k^*-(k^*-1)\gamma},$$ 
$Y_1(n) = n - \sum_{k=2}^AY_k(n)$, and 
$$T(n) = (n-\sum_{k=2}^A(k-1)Y_k(n))^\gamma +
\sum_{k=1}^A w(k)Y_k(n).$$ 

More precisely, the procedure is to sequentially compute, for each $\ell \geq 0$ fixed, the values $c^k_{k+\ell}$ as $k$ runs through $2\leq k< \gamma/(\gamma-1)$, with the provision that $k+\ell\leq k^*$.

In section \ref{sec:algebra}, we reprise this procedure to derive the coefficients, and work through an illustrative example.

To show the central limit theorem, when $2\leq k<\gamma/(\gamma-1)$, we need to show that 
$$\frac{M_k(n)}{n^{(k-(k-1)\gamma)/2}} \Rightarrow N(0, a_k).$$
  Let $\xi^{(n)}_j = n^{-(k-(k-1)\gamma)/2}\big\{d_k(j+1) - E[d_k(j+1)|\mathcal{F}_j]\}$, for $1\leq j\leq n-1$, define increments of the martingale array $\{M_k(j)/n^{(k-(k-1)\gamma)/2}: 1\leq j\leq n-1, n\geq 2\}$ with respect to sigma-fields $\{\mathcal{F}^{(n)}_j = \mathcal{F}_j: 1\leq j\leq n-1, n\geq 2\}$.  
  
  Note that, a.s.,
  $$\max_{1\leq j\leq n-1}|\xi^{(n)}_j|\leq n^{-(k-(k-1)\gamma)/2} \rightarrow 0.$$
    Also, the conditional variance, 
$$\sum_{j=1}^{n-1}E \big[(\xi^{(n)}_j)^2|\mathcal{F}_j\big] = \frac{1}{n^{k-(k-1)\gamma}}\langle M_k\rangle(n).$$
Since $S(j)\approx j^\gamma$ (Proposition \ref{S_asymptotics}) and $Z_\ell(j)\approx a_\ell j^{\ell - (\ell-1)\gamma}$ when $1\leq \ell \leq k$ (Theorem \ref{degree_thm}),
noting the quadratic variation formulas \eqref{quad_comp}, the conditional variance is of form $n^{-(k-(k-1)\gamma)}\sum_{j=1}^{n-1}\frac{w(k-1)Z_{k-1}(j)}{S(j)} + o(1)$, which converges a.s. to $w(k-1)a_{k-1}/(k-(k-1)\gamma) = a_k$. 

 Hence, when $2\leq k<\gamma/(\gamma-1)$, the standard Lindeberg assumptions for the martingale central limit theorem, Corollary 3.1 in \cite{HallHeyde}, are clearly satisfied, and the desired Normal convergence holds.  
 
 \medskip
When $2\leq k= \gamma/(\gamma -1)$ is an integer, the theorem statement follows from an easier version of the same argument.  Indeed, after initial decompositions, noting that $j^{-\gamma}Z_{k-1}(j) = j^{(k-1)(1-\gamma)} = j^{-1}$ by Theorem \ref{degree_thm}, 
\begin{eqnarray*}
Z_k(n) &=& \sum_{j=1}^{n-1} \frac{w(k-1)Z_{k-1}(j)}{S(j)} + M_k(n)\\
& =& b_k\log(n) + O(1) + M_k(n).
\end{eqnarray*}
  The quadratic variation of $(\log(n))^{-1/2}M_k(n)$, inspecting formulas \eqref{quad_comp}, converges a.s. to $b_k$.  
  The Normal convergence follows now as above.  

\medskip
When $\gamma/(\gamma-1)<2$, that is when $\gamma>2$, 
$$n - Z_1(n) = -2+\sum_{j=1}^{n-1} \frac{w(1)Z_1(j)}{S(j)} - M_1(n).$$
The sum converges as the summand is of order $n^{1-\gamma}$.  Since $n-Z_1(n)\geq 0$, the bounded difference martingale $M_1(n)$ must also converge, as its liminf cannot be $-\infty$ (cf. \eqref{bounded_diff}). Hence, $Z_1(n)-n$ converges a.s. 

Similarly, note (i) the remarks below \eqref{durrett_comment}, (ii) the maximum $V(n)$ is achieved at a fixed (random) vertex $v$ and other vertices have bounded degree (Proposition \ref{frozen}), and (iii) $\Phi_{A+1}(n)$ converges (Theorem \ref{degree_thm}) and so the set of corresponding vertices stabilizes after a random time to a finite set $\hat\Phi_{A+1}(\infty)$.  Then, for all large $n$, a.s.
\begin{equation*}
n-V(n)  = -1  + \sum_{k=2}^A(k-1)Z_k(n) \ + \text{weight  of  vertices in }\hat\Phi_{A+1}(\infty)\setminus\{v\}.
 \end{equation*}
 By Theorem \ref{degree_thm}, as $\gamma/(\gamma-1)<2$, the sum in the above display converges a.s. as $n\uparrow\infty$.  Hence, $V(n)-n$ converges a.s.
  \qed

\medskip
\noindent
{\bf Proof of Corollary \ref{mainthm3}.}   
As in the proof of Theorem~\ref{mainthm2}, we recast \eqref{conservation} as
\begin{align*}
Z_1(n) & = n - Z_2(n) - Z_3(n) - \cdots -Z_A(n) +O(1)\\
V(n) & = n - Z_2(n) - 2 Z_3(n) -\cdots -(A-1) Z_A(n) + O(1).  
\end{align*}
We now argue in the setting when $\gamma/(\gamma-1)>2$, as the case $\gamma/(\gamma-1)=2$ is similar and easier.

Since the order of the conditional variance of $M_2(n)$ is $n^{1-\gamma/2}$, and the martingales $M_k(n)$ are of lower order for $k>2$, 
we have 
\begin{align}
\label{b_decomp}
& \lim_{n \uparrow \infty} \frac{1}{n^{1-\gamma/2}}\Big(Z_1(n) + M_2(n) - n +   \sum_{k=2}^{2^*} \sum_{\ell = k}^{2^*} c_\ell^kn^{\ell - (\ell-1)\gamma}\Big)    = 0 \ \ a.s. \\
& \lim_{n \uparrow \infty} \frac{1}{n^{1-\gamma/2}}\Big(V(n) + M_2(n) - n + \sum_{k=2}^{2^*} \sum_{\ell = k}^{2^*} (k-1)c_\ell^k n^{\ell - (\ell-1)\gamma}\Big)    = 0 \ \ a.s. \nonumber
\end{align}
and Corollary~\ref{mainthm3} follows with $c^1_\ell =  \sum_{k=2}^{2^*} c^k_\ell$ and $m_\ell = \sum_{k=2}^{2^*} (k-1)c^k_\ell$.  \qed

\section{Example}
\label{sec:algebra}

We first recall the procedure that yields asymptotic expansions for the solutions of~\eqref{c_decomp}
when $\frac{\gamma}{\gamma-1} \geq 2$, in other words, $1 < \gamma \leq 2$.  Recall also that $A = \lfloor \frac{\gamma}{\gamma-1}\rfloor$.  

Define gauge sequences (or asymptotic scales) $\varphi_\ell(j)$, for  $0 \leq \ell < A$, by
$$
\varphi_\ell(j) = \begin{cases} j^{\ell +1- \ell \gamma} & 0 \leq \ell < \frac{1}{\gamma-1} \\
\log(j) & \ell = A-1 = \frac{1}{\gamma-1}. \end{cases}
$$
Note that, $\varphi_\ell(j)$ is diverging for $l=0,1,2,\ldots,A-1$, $\lim_{j \uparrow \infty} \frac{\varphi_{\ell+1}(j)}{\varphi_\ell(j)} = 0$.  We get $\phi_{A-1}(j) = \log(j)$ only if $\frac{\gamma}{\gamma -1}$ is an integer. It is easy to verify that
$$
\sum_{j=1}^{n-1} \frac{\varphi_{\ell-1}}{j^\gamma} \approx \begin{cases} (\ell+1-\ell \gamma)^{-1} \varphi_{\ell}(n) + O(1) & 1 \leq \ell < \frac{1}{\gamma-1} \\ \varphi_{A-1} + O(1) & \ell = \frac{1}{\gamma-1} \\  O(1) & \ell = A. \end{cases}
$$
Consequently, substituting asymptotic expansions for $Y_k$ up to $O(\varphi_{\ell-1})$ in the right hand side of \eqref{c_decomp} will generate the corresponding asymptotic expansions up to $O(\varphi_{\ell})$ on the  left hand side, and we can iterate this procedure to find the complete asymptotic expansions of the solutions of~\eqref{c_decomp}.

In expanding 
\eqref{c_decomp}, we need only retain the diverging, non-random terms, that are on scales equal to  or larger than the fluctuations given by the martingales $M_k$, to obtain the expansions in Theorem~\ref{mainthm2} and Corollary~\ref{mainthm3}. The quadratic variation computation in \eqref{quadvar} yields
$$
\langle M_k \rangle (n) =O\left(\varphi_{k-1}(n)\right),
$$
so that by \eqref{fisher}, we need to retain terms up to $O(\varphi_{k^*-1})$, where
$$
k^* = \left\lfloor \frac{\gamma}{2(\gamma-1)} + \frac{k}{2} \right\rfloor =  \max\Big\{\ell  \ \Big | \ \lim_{j \uparrow \infty} \frac{\sqrt{\varphi_{k-1}(j)}}{\varphi_{\ell-1}(j)} < \infty \Big\} \ \ \mbox{ for } k=2,3,\ldots,A.
$$ 
We remark, in the case $k=\gamma/(\gamma-1)$ is the integer $A$, that $k^* =k$.

For $\ell= 0,1,2,\ldots$, sequentially compute $c^k_{k+\ell}$ as $k$ runs through $2\leq k\leq A$, with the provision that $k+\ell\leq k^*$.  If we arrange the coefficients $c^k_j$ in a rectangular array with row index $k$ and column index $j$, this procedure corresponds to first computing the main diagonal, then the first super-diagonal, and then the second and so on.

\medskip

We now illustrate this procedure by considering the explicit case $\gamma = 5/4$. This corresponds to when $\gamma/(\gamma-1) = 5$, an integer. The corresponding gauge sequences are $\varphi_0(j) = j, \varphi_1(j) = j^{3/4}, \varphi_2(j) = j^{1/2}, \varphi_3(j) = j^{1/4}, \varphi_4(j) = \log(j)$.
Retaining terms up to the order of the square root of $\varphi_{k-1}$, for $2\leq k \leq 5$, we posit expansions,
\begin{align*}
Y_2(j) & = c_2^2 j^{3/4} + c_3^2 j^{1/2} + o(j^{1/2}),\\
Y_3(j) & = c_3^3  j^{1/2} + c_4^3 j^{1/4} +  o(j^{1/4}), \\
Y_4(j) & = c_4^4 j^{1/4} + o(j^{1/4}), \\
Y_5(j) & = c_5^5 \log(j)  + o(\log(j)).
\end{align*}
Although we know $c_k^k = a(k)$, from the asymptotics in Theorem~\ref{degree_thm}, we will however compute them as part of the procedure as a check on the method.

For $\ell =0, k = 2$, we have $Y_2(j) = c_2^2 j^{3/4} + o(j^{3/4})$, $Y_3(j) = O(j^{1/2})$, $Y_4(j) = O(j^{1/4})$, $Y_5(j) = O(\log(j))$, so that $Y_1 = j -   c_2^2 j^{3/4} +  O(j^{1/2})$ and $T(j) = (j- c_2^2 j^{3/4} + O(j^{1/2}))^\gamma$.

Equation \eqref{c_decomp} now yields
\begin{align*}
c_2^2 n^{3/4} + o(n^{3/4}) & = \sum_{j=1}^{n-1} \frac{j - (1+2^\gamma) c_2^2 j^{3/4} + O(j^{1/2})}{(j - c_2^2 j^{3/4} + O(j^{1/2}))^{5/4} + j + (2^\gamma-1) c_2^2 j^{3/4} + O(j^{1/2})} \\
& = \frac{4}{3} n^{3/4} + o(n^{3/4}) 
\end{align*}
where we obtain the second line by expanding the summand as a series in $j$ about $j = \infty$. Consequently $c_2^2 = \frac{4}{3} = \frac{w(1)}{2-\gamma}$ in agreement with Theorem~\ref{degree_thm}. 

Applying the same method, sequentially, with $k=3,4$ and $5$ yields
 \begin{align*}
&c_3^3 n^{1/2} + o(n^{1/2})  = \sum_{j=1}^{n-1} \frac{\frac{4}{3} 2^\gamma j^{3/4} + O(j^{1/2})}{(j - c_2^2 j^{3/4} + O(j^{1/2}))^{5/4} + j + (2^\gamma-1) c_2^2 j^{3/4} + O(j^{1/2})} \\
&\ \ \ \ \ \ \ \ \ \ \ \ \ \ \ \ \ \ \ \ \ = \frac{8}{3} 2 ^\gamma n^{1/2} + o(n^{1/2}) \\
&c_4^4 n^{1/4} + o(n^{1/4})  = \sum_{j=1}^{n-1} \frac{\frac{8}{3} 2^\gamma 3^\gamma  j^{1/2} + O(j^{1/4})}{(j - c_2^2 j^{3/4} + O(j^{1/2}))^{5/4} + j + (2^\gamma-1) c_2^2 j^{3/4} + O(j^{1/2})} \\
&\ \ \ \ \ \ \ \ \ \ \ \ \ \ \ \ \ \ \ \ \ = \frac{32}{3} 6^\gamma n^{1/4} + o(n^{1/4}) 
\end{align*}
\begin{align*}
& c_5^5 \log(n) + o(\log(n)) \\
&\ \ \ \ \ \ \ \  = \sum_{j=1}^{n-1} \frac{\frac{32}{3} 24^\gamma  j^{1/4} + o(j^{1/4})}{(j - c_2^2 j^{3/4} + O(j^{1/2}))^{5/4} + j + (2^\gamma-1) c_2^2 j^{3/4} + O(j^{1/2})} \\
&\ \ \ \ \ \ \ \ = \frac{32}{3} 24^\gamma \log(n) + o(\log(n)) 
\end{align*} 
These values for $c_3^3,c_4^4$ and $c_5^5$ also agree with the conclusions of Theorem~\ref{degree_thm}, namely,
$$
c_3^3 = a_3= \frac{w(2)w(1)}{(2-\gamma)(3-2\gamma)}, \quad c_4^4 = a_4= \frac{w(3)w(2)w(1)}{(2-\gamma)(3-2\gamma)(4-3\gamma)}, 
$$ 
and $c_5^5 = b_5 = w(4) a_4$. 

Using these values, we have the refined expansions:
\begin{align*}
Y_2(j) & = \frac{4}{3} j^{3/4} + c_3^2 j^{1/2} + o(j^{1/2}), \\
Y_3(j) & = \frac{8}{3} 2^\gamma j^{1/2} + c_4^3 j^{1/4} + o(j^{1/4}) \\
Y_4(j) & = \frac{32}{3} 6^\gamma j^{1/4} + o(j^{1/4}), \\
Y_5(j) & = \frac{32}{3} 24^\gamma \log(j) + o(\log(j)).
\end{align*}
Now, we set $\ell = 1$ and we again iterate over $k$. 
For $k = 2$, \eqref{c_decomp} yields
\begin{align*}
\frac{4}{3} n^{3/4} + c_3^2 n^{1/2} + o(n^{1/2}) & = \sum_{j=1}^{n-1} \frac{j - \frac{4}{3}(1+2^\gamma) j^{3/4} +O(j^{1/2})}{(j - \frac{4}{3} j^{3/4} +O(j^{1/2}))^{5/4} + j + O(j^{3/4})} \\
& = \frac{4}{3} n^{3/4} - \left(\frac{4}{3} - \frac{8}{3}2^\gamma\right) n^{1/2} + O(n^{1/4}) 
\end{align*}
showing that $c^2_3 = - \left(\frac{4}{3} - \frac{8}{3}2^\gamma\right)$ and all deterministic remainder terms in the expansion of $Y_2(j)$ are actually $O(j^{1/4})$. 
For $k = 3$, we get
\begin{align*}
\frac{8}{3} n^{1/2} + c_4^3 n^{1/4} + o(n^{1/4}) & = \sum_{j=1}^{n-1} \frac{c^2_3 2^\gamma j^{3/4} - \left(\frac{4}{3} 2^\gamma + \frac{8}{3} 6^\gamma\right)  j^{1/2} + o(j^{1/2})}{(j - \frac{4}{3} j^{3/4} +O(j^{1/2}))^{5/4} + j + O(j^{3/4})}  \\
& = \frac{8}{3}2^\gamma n^{1/2} -\frac{16}{9}  \left(2^\gamma +  6\cdot 4^\gamma +6\cdot 6^\gamma\right) n^{1/4} + O(\log(n)) \\
\implies \quad c_4^3 & = -\frac{16}{9} \left(2^\gamma +  6\cdot 4^\gamma +6\cdot 6^\gamma\right).
\end{align*}

Consequently, for $\gamma = \frac{5}{4}$, from Theorem~\ref{mainthm2}, Corollary~\ref{mainthm3} and~\eqref{b_decomp}, we have the central limit theorems
\begin{align}
&  \lim_{n \uparrow \infty} \frac{1}{n^{3/8}}  \begin{pmatrix}Z_1(n) - n + \frac{4}{3}n^{3/4}+\left(\frac{16}{3}2^\gamma- \frac{4}{3} \right)n^{1/2}\\ -Z_2(n) + \frac{4}{3}n^{3/4}+\left(\frac{8}{3}2^\gamma- \frac{4}{3} \right)n^{1/2}\\ V(n)-n + \frac{4}{3}n^{3/4}+\left(8\cdot2^\gamma- \frac{4}{3} \right)n^{1/2} \end{pmatrix} \Rightarrow N\left(\begin{pmatrix} 0 \\ 0 \\ 0\end{pmatrix}, \frac{4}{3}\begin{pmatrix} 1 & 1 & 1\\ 1 & 1 & 1 \\ 1 & 1 & 1\end{pmatrix} \right)  \nonumber \\
&\lim_{n \uparrow \infty} \frac{1}{n^{1/4}}\Big(Z_3(n)-\frac{8}{3}2^\gamma n^{1/2} +\frac{16}{9}  \left(2^\gamma +  6\cdot 4^\gamma +6\cdot 6^\gamma\right)  n^{1/4}\Big)  \Rightarrow N\left(0, \frac{8}{3}2^\gamma \right) \nonumber \\
&\lim_{n \uparrow \infty} \frac{1}{n^{1/8}}\Big(Z_4(n)-\frac{32}{3}6^\gamma n^{1/4} \Big)  \Rightarrow N\left(0, \frac{32}{3}6^\gamma \right) \nonumber \\
& \lim_{n \uparrow \infty} \frac{1}{\sqrt{\log(n)}}\Big(Z_5(n) - \frac{32}{3}24^\gamma \log(n)\Big)   \Rightarrow N\left(0, \frac{32}{3}24^\gamma \right).
\label{fluctuations}
\end{align}

 \vskip .2cm
\noindent {\bf Acknowledgement.}  This work was partially supported by ARO W911NF-14-1-0179.


\vskip .1cm 

\begin{thebibliography}{99}


\bibitem{Albert-Barabasi-02}
{ Albert, R. and Barab\'{a}si, A.-L.} (2002).
\newblock Statistical mechanics of complex networks.
\newblock {\it Rev. Modern Phys.\/} {\bf 74,} 47--97.


\bibitem{Athreya}
{ Athreya, K. B.} (2007).
\newblock Preferential attachment random graphs with general weight function.
\newblock {\it Internet Math.\/} {\bf 4,} 401--418.

\bibitem{AGS}
{ Athreya, K. B., Ghosh, A. P. and Sethuraman S.} (2008).
\newblock Growth of preferential attachment random graphs via continuous-time branching processes.
\newblock {\it Proc. Indian Acad. Sci. Math. Sci.\/} {\bf 118,} 473--494.





\bibitem{BRST-01}
{ Bollob\'{a}s, B., Riordan, O., Spencer, J. and Tusn\'{a}dy, G.} (2001).
\newblock The degree sequence of a scale-free random graph process.
\newblock {\it Random Structures Algorithms\/} {\bf 18,} 279--290.



\bibitem{Cald}
{ Caldarelli, G.} (2007).
\newblock {\it Scale-Free Networks: Complex Webs in Nature and Technology}.
\newblock Oxford University Press, USA.


\bibitem{CS}
{ Choi, J. and Sethuraman, S.} (2011)
\newblock Large deviations for the degree structure in preferential attachment schemes.
\newblock {\it Ann. Appl. Probab.\/} {\bf 23,} 722--763.

\bibitem{CSV}
{ Choi, J., Sethuraman, S., Venkataramani, S. C.} (2015)
\newblock A scaling limit for the degree distribution in sublinear preferential attachment schemes
\newblock {\it Rand. Struct. Alg.\/} {\bf 48,} 703--731.

\bibitem{CHJ}
{ Chung F., Handjani, S. and Jungreis, D.} (2003).
\newblock Generalizations of P\'olya's urn problem.
\newblock {\it Annals of Combinatorics\/} {\bf 7,} 141--153.




\bibitem{Chung-Lu}
{ Chung, F. and Lu, L.} (2006).
\newblock {\it Complex Graphs and Networks} vol.~107 of {\it CBMS Regional
  Conference Series in Mathematics}.
\newblock Published for the Conference Board of the Mathematical Sciences,
  Washington, DC.


\bibitem{DM}
{ Dorogovtsev, S. N. and Mendes, J. F. F.} (2003).
\newblock {\it Evolution of Networks: From Biological Nets to the Internet and WWW}.
\newblock Oxford University Press, USA


\bibitem{Durrett_prob}
{ Durrett, R.} (2010).
\newblock {\it Probability:  Theory and Examples} 4th Ed.
\newblock Cambridge U. Press.

\bibitem{Durrettbook}
{Durrett, R.} (2007).
\newblock {\it Random Graph Dynamics}.
\newblock Cambridge U. Press.

\bibitem{Fisher}
{ Fisher, E. } (1992)
\newblock On the law of the iterated logarithm for martingales
\newblock {\it Ann. Probab.\/} {\bf 20}  675--680.

\bibitem{HallHeyde}
{ Hall, P. and Heyde, C.C.} (1990)
\newblock {\it Martingale Limit Theory and its Application}
\newblock Academic Press, New York.

\bibitem{KR}
{ Krapivsky, P. and Redner, S.} (2001).
\newblock Organization of growing random networks.
\newblock {\it Phys. Rev. E\/} {\bf 63} 066123-1 -- 066123-14.

\bibitem{KRL}
{ Krapivisky, P., Redner, S., and Leyvraz, F.} (2000)
\newblock Connectivity of Growing Random Networks.
\newblock {\it Phys. Rev. Lett.\/} {\bf 85} 4629 -- 4632.


\bibitem{Mitzenmacher}
{ Mitzenmacher, M.} (2004).
\newblock A brief history of generative models for power law and
lognormal distributions.
\newblock {\it Internet Math.} {\bf 1,} 226--251.


\bibitem{Mori-01}
{ M\'{o}ri, T.~F.} (2002).
\newblock On random trees.
\newblock {\it Studia Sci. Math. Hungar.\/} {\bf 39,} 143--155.



\bibitem{Newman10}
{ Newman, M.~E.~J.} (2010).
\newblock {\it Networks: An Introduction}.
\newblock Oxford University Press, USA.



\bibitem{Oliveira}
{ Oliveira, R. and Spencer, J.} (2005).
\newblock Connectivity transitions in networks with super-linear preferential attachment.
\newblock {\it Internet Math.\/} {\bf 2,} 121--163.


\bibitem{RTV}
{ Rudas, A., T\'oth, B. and Valk\'o, B.} (2007).
\newblock {Random trees and general branching processes.}
\newblock {\it Random Struct. Algorithms\/} {\bf 31,} 186--202.

\end{thebibliography}
\end{document}